\documentclass [11pt,oneside]{amsart}
\usepackage[leqno]{amsmath}
\usepackage{amsthm}
\usepackage{dsfont}
\usepackage[euler]{textgreek}
\usepackage{amsfonts}
\usepackage{amssymb}
\usepackage{eucal}
\usepackage{upgreek}
\usepackage{extpfeil}
\usepackage{latexsym,amscd,fullpag e, tikz, cmll, mathrsfs, epic, mathtools} 
\usepackage{color}
\usepackage{graphicx}
\usepackage [all]{xy}
\usepackage{hyperref}
\usetikzlibrary{fit}
\usetikzlibrary{arrows}
\usetikzlibrary{matrix}
\SelectTips{cm}{12}

  \newtheorem{theorem}{Theorem}[subsection]
  \newtheorem{proposition}[theorem]{Proposition}
  \newtheorem{lemma}[theorem]{Lemma}
  \newtheorem{corollary}[theorem]{Corollary}

\theoremstyle{definition}
  \newtheorem{remark}[theorem]{Remark}
  \newtheorem{example}[theorem]{Example}

\theoremstyle{remark}

\theoremstyle{definition}
  \newtheorem{definition}[theorem]{Definition} 
  \newtheorem{question}[theorem]{Question}

\renewcommand {\geq} {\geqslant}

\def\1{{\bf 1}}

\newextarrow{\xbigtoto}{{20}{20}{20}{20}}
   {\bigRelbar\bigRelbar{\bigtwoarrowsleft\rightarrow\rightarrow}}

\DeclareSymbolFont{largesym}{OML}{cmm}{m}{it}
\DeclareMathSymbol{\nstnsmall}{0}{largesym}{"22}

\title{Rational connectivity and \\ analytic contractibility}
\author{Morgan Brown }
\address{{\bf Morgan Brown}\newline Department of Mathematics\\
University of Miami\\
Coral Gables, FL 33146 USA}
\email{mvbrown@math.miami.edu}
\author{Tyler Foster}
\address{{\bf Tyler Foster}\newline Department of Mathematics\\
University of Michigan\\
Ann Arbor, MI 48109, USA}
\email{tyfoster@umich.edu}
\date {\today}
\subjclass[2010]{Primary 32P05, 32C18; Secondary 14E30}
\keywords{non-archimedean analytic geometry, Berkovich spaces, rationally connected varieties, minimal model program}

\begin{document}

\begin{abstract}
Let $k$ be an algebraically closed field of characteristic $0$, and let $f:X \longrightarrow Y$ be a morphism of smooth projective varieties over the ring $k((t))$ of formal Laurent series. We prove that if a general geometric fiber of $f$ is rationally connected, then the induced map $f^{\mathrm{an}}:X^{\mathrm{an}}\longrightarrow Y^{\mathrm{an}}$ between the Berkovich analytifications of $X$ and $Y$ is a homotopy equivalence. Two important consequences of this result are that the Berkovich analytification of any $\mathbb{P}^{n}$-bundle over a smooth projective $k((t))$-variety is homotopy equivalent to the Berkovich analytififcation of the base, and that the Berkovich analytification of a rationally connected smooth projective variety over $k((t))$ is contractible.
\end{abstract}

\maketitle

\section{Introduction}\label{intro}
	Throughout this paper, a {\em variety} over a field $K$ is any integral, separated, finite type $K$-scheme. Fix an algebraically closed field $k$ of characteristic $0$ once and for all. Given a $k$-scheme $S$, we say that $X\longrightarrow S$ is a variety over $S$ if $X$ is an integral, separated scheme of finite type scheme over $S$. We use the term \emph{generic point} to mean the point of the variety $X$ corresponding to its field of rational functions, while a \emph{general point} refers to any closed point chosen inside some open set. If $T$ is a second $S$-scheme, then a \emph{general $T$-point} of $X$ refers to any $T$-valued point $T\longrightarrow X$ whose image is a general point.

\vskip .5cm

\subsection{A suggestive dictionary} Let $K$ be a topological field, topologically complete with respect to a  fixed norm $|-|:K\longrightarrow\mathbb{R}_{\ge0}$. Assume that this norm satisfies the ultrametric triangle inequality $|a+b|\le\mathrm{max}\{|a|,|b|\}$ for all $a,b\in K$. We call a field $K$ complete with respect to such a norm a {\em non-archimedean field}.

In close analogy with the operation of complex analytification, Berkovich showed how to associate a well behaved topological space $X^{\mathrm{an}}$ to each finite type $K$-scheme $X$ \cite[\S 3.4]{Berk}. The space $X^{\mathrm{an}}$ is called the {\em Berkovich analytification} of $X$. Berkovich observed that there exists a partial dictionary between algebro-geometric properties of $X$ and topological properties of $X^{\mathrm{an}}$. Namely:

\vskip .25cm

{\bf Theorem  \cite[Theorem 3.4.8]{Berk}.} {\em Let $X$ be a scheme locally of finite type over $K$. Then:
\begin{itemize}
\item[{\bf (i)}] $X$ is separated $\iff$ $X^{\mathrm{an}}$ is Hausdorff.
\item[{\bf (ii)}] $X$ is proper $\iff$ $X^{\mathrm{an}}$ is Hausdorff and compact.
\item[{\bf (iii)}] $X$ is connected $\iff$ $X^{\mathrm{an}}$ is pathwise connected.
\item[{\bf (iv)}] The algebraic dimension of $X$ $=$ the topological dimension of $X^{\mathrm{an}}$.
\vskip .25cm
\end{itemize}}

\vskip .25cm

\noindent
Naturally, one would like to continue Berkovich's dictionary, and so we ask:

\begin{question}\label{giant question} Which other algebro-geometric properties of $X$ correspond to topological properties of $X^{\mathrm{an}}$?
\end{question}

One fundamental invariant of a topological space is its homotopy type. Our main theorem in the present article provides a tool for relating the birational geometry of $K$-varieties to the homotopy types of their Berkovich analytifications, in the special case that $K=k((t))$, the field of formal Laurent series over an algebraically closed field $k$ of characteristic $0$:

\begin{theorem}\label{main}
Let $k$ be an algebraically closed field of characteristic $0$, and let $f: X \longrightarrow Y$ be a surjective map of smooth projective varieties over the nonarchimedean field $K= k((t))$. Suppose that for a general geometric point $q$ of  $Y$ the fiber $X_q$ is rationally connected. Then the induced map of Berkovich spaces $f^{\mathrm{an}}:X^{\mathrm{an}}\longrightarrow Y^\mathrm{{an}}$ is a homotopy equivalence.
\end{theorem}

As immediate corollaries, we have

\begin{corollary}\label{pb}
If $Y$ is a smooth $k((t))$-variety and $f:X\longrightarrow Y$ is a $\mathbb{P}^{n}$-bundle over $Y$, then $f^{\mathrm{an}}:X^{\mathrm{an}}\longrightarrow Y^{\mathrm{an}}$ is a homotopy equivalence.
\end{corollary}

\begin{corollary}\label{rc}
Let $X$ be a smooth projective $k((t))$-variety. If $X$ is rationally connected, then its Berkovich analytification $X^{\mathrm{an}}$ has the homotopy type of a point, i.e., is contractible.
\end{corollary}

Berkovich showed that for projective space $\mathbb{P}_{\!K}^{n}$, of any dimension $n$ over any non-archimedean field $K$, the analytification $(\mathbb{P}_{\!K}^{n})^{\mathrm{an}}$ is contractible \cite[Theorem 6.1.5]{Berk}. The simplest example of a rationally connected variety is projective $n$-space, so Corollary \ref{rc} can be seen as a broad extension of Berkovich's contractibility result in the special case that $K=k((t))$, and Corollary \ref{pb} can be seen as an extension of Berkovich's result to the relative setting.

Finally, given $X$ and $Y$ smooth, projective and birational, we may choose a smooth variety $Z$ admitting surjective maps to both $X$ and $Y$, so that Theorem \ref{main} implies Musta{\c{t}}{\u{a}} and Nicaise's result that the homotopy type of the Berkovich analytification is a birational invariant  \cite{MN}, \cite[2.4.13]{Nic2}.

Recent work of Musta{\c{t}}{\u{a}} and  Nicaise  \cite{MN} and Nicaise and Xu \cite{NX} explores connections between the minimal model program (MMP) and the geometry of Berkovich spaces. We adopt many of their techniques in the proof of Theorem \ref{main}. We also use recent results by de Fernex, Koll\'ar, and Xu on the homotopy type of the dual complex of a resolution of singularities \cite{dFKX}. It is our use of these results that forces us to work over the non-archimedean field $k((t))$, with $k$ algebraically closed of characteristic $0$. It remains an open question  whether or not a counterpart to Theorem \ref{main}, or either of its Corollaries \ref{pb} or \ref{rc}, holds over more general non-archimedean fields.

\subsection{Organization of the paper} In Section \ref{skeleta}, we briefly review Berkovich analytifications and their skeleta. We use results of \cite{MN} and \cite{NX} to construct a skeleton $\mathrm{Sk}(\mathcal{X})\subset{X^{\mathrm{an}}}$, and we reduce the proof of Theorem \ref{main} to a verification that the map $f:X\longrightarrow Y$ appearing in the statement of the theorem induces a homotopy equivalence $\mathrm{Sk}(\mathcal{X})\simeq\mathrm{Sk}(\mathcal{Y})$

In Section \ref{rcvarieties}, we review rational connectivity of algebraic varieties in both the absolute and relative settings.

The main technical contribution of the paper occurs in Section \ref{spreading out section}. Here we show that if $\mathcal{X}$ is an snc $k[[t]]$-model of a rationally connected variety, then we can construct a degenerating family of rationally connected varieties over an algebraic curve, such that the dual complex of the special fiber of this degenerating family is homeomorphic to our skeleton $\mathrm{Sk}(\mathcal{X})$. We actually need a relative version of this statement in order to prove Theorem \ref{main}, and it is the relative version that we prove in Section \ref{spreading out section}. This puts us in a situation where the results of \cite{dFKX} become applicable.

In Section \ref{applying dFKX}, we establish a relative version of \cite[Theorem 41]{dFKX}.  As a consequence, we are able to show that the special fiber of a snc degeneration of rationally connected varieties always has a rationally connected component, giving an affirmative answer to a question of Koll\'ar in \cite[Question 11]{Kol2}. We then complete the proof of Theorem \ref{main}.


\subsection{Thanks}The authors would like to thank Matt Baker, Mattias Jonsson, Eric Katz, Mircea Musta{\c{t}}{\u{a}}, James McKernan, Sam Payne, Joe Rabinoff, Chenyang Xu, and David Zurieck-Brown for invaluable discussions over the course of writing this article. Several of these discussions occurred at the BIRS workshop ``Specialization of Linear Series for Algebraic and Tropical Curves'', at which the first author was a participant.  We would like to thank Sam Payne in particular for remarks on an earlier draft of this paper. His remarks led to a significant strengthening of our results.
We would also like to thank the reviewers. Their remarks greatly improved this paper, and brought our attention to a gap in the previous draft.
	
Both authors were partially supported by NSF RTG grant DMS-0943832 during the completion of this work. The first author was also supported by NSF grants DMS-1200656 and DMS-1265263 for part of the completion of this work. The first author would like to thank UCSD for hosting him as a visitor during part of the completion of this work.
\section{Skeleta in Berkovich spaces}\label{skeleta}

\subsection{The Berkovich analytification} Let $K$ be a field complete with respect to a non-archimedean norm $|-|:K\longrightarrow\mathbb{R}_{\ge0}$. Assume that this norm is {\em discrete}, in the sense that its value group $|K^{\times}|$ is a discrete subgroup of $\mathbb{R}_{\ge0}$. Let $R=\{a\in K:|a|\le1\}$ denote the ring of integers in $K$, with unique maximal ideal $\mathfrak{m}=\{a\in K:|a|<1\}$, and denote its residue field $k\overset{{}_{\mathrm{def}}}{=}R/\mathfrak{m}$.

Let $X$ be a proper $K$-variety, and let $X^{\mathrm{an}}$ denote its Berkovich analytification. A variety for us will always be separated, and hence its Berkovich analytification will always be a Hausdorff topological space. In this article, we consider $X^{\mathrm{an}}$ solely as a topological space, completely ignoring the analytic structure sheaf on $X^{\mathrm{an}}$.

Let us briefly recall that if $U=\mathrm{Spec}_{\ \!}A$ is an affine local chart on $X$, then $U^{\mathrm{an}}$ is an open subset of $X^{\mathrm{an}}$ given by
$$
U^{\mathrm{an}}
\ =\ 
\left\{
\begin{array}{c}
\mbox{multiplicative seminorms }
\\
|-|_{x}:A\!\longrightarrow\!\mathbb{R}_{\ge0}\mbox{\ that restrict}
\\
\mbox{to the norm }|-|\mbox{ on }K\!\subset\!A
\end{array}
\right\}.
$$
Each function $f\in A$ determines a map $\mathrm{ev}_{f}:U^{\mathrm{an}}\longrightarrow\mathbb{R}_{\ge0}$ that takes $|-|_{x}\mapsto |f|_{x}$. The topology on $U^{\mathrm{an}}$ is the coarsest topology that renders each of these maps continuous. We glue $X^{\mathrm{an}}$ from these topological spaces $U^{\mathrm{an}}$, and the resulting topological space is independent of the particular affine cover we use.

\subsection{Skeleta} The Berkovich analytification tends to be large and complicated. For example, the point set underlying the Berkovich analytification of a curve has a natural structure of an infinite graph. One powerful technique for understanding a given Berkovich analytification is to produce a smaller space inside $X^{\mathrm{an}}$, called a {\em skeleton}, which reflects much of the geometry of the ambient space.  Specifically, this skeleton is supposed to be a finite simplicial complex onto which $X^{\mathrm{an}}$ deformation retracts. In the case of curves, the skeleton turns out to be a finite graph. 

A general technique for producing a skeleton is to choose a model $\mathcal{X}$ of $X$ over the local ring $R$ of integers in $K$, such that the special fiber $\mathcal{X}_{k}$ is a simple normal crossing divisor in $\mathcal{X}$. Each prime component of the special fiber corresponds to a point in ${X^{\mathrm{an}}}$, and the full dual complex  of $\mathcal{X}_{k}$ embeds as a closed subspace of ${X^{\mathrm{an}}}$. Note though that because this construction requires the choice of a model $\mathcal{X}$, it is not canonical.

It is often desirable to produce a skeleton that is canonical in the sense that it can be recovered uniquely from $X$. One construction which produces such a skeleton, called the {\em essential skeleton}, is due to Musta{\c{t}}{\u{a}} and Nicaise \cite{MN}, following work of Kontsevich and Soibelman \cite{KS} on mirror symmetry. Their construction uses weight functions coming from pluricanonical forms on $X$. Nicaise and Xu \cite{NX} use this construction to study degenerations of Calabi-Yau varieties and, more generally, smooth projective varieties with semi-ample canonical bundle.  When $X$ is a Calabi-Yau variety, one can use an approximation argument to show that the the essential skeleton is a strong deformation retract of $X^{\mathrm{an}}$. If $X$ is of general type, and if the canonical divisor $K_X$ is semi-ample, one expects to be able to produce a minimal model $\mathcal{X}$ over $R$, with $K_{\mathcal{X}}$ semi-ample. This model will be unique up to simple birational modifications called flops. Nicaise and Xu show that when $X$ is smooth and projective with semi-ample canonical bundle, the skeleton constructed in \cite{MN} is given by the dual complex of any minimal dlt model of $X$. When $X$ is defined over an algebraic curve \cite[Theorem 2.2.6]{NX}, one can use the MMP to produce a minimal dlt model model of $X$ over $R$. In their recent preprint \cite{KNX}, Koll\'ar, Nicaise, and Xu show that a minimal dlt model of $X$ always exists over $R$, regardless of whether or not $X$ is defined over a curve.

On the other hand, a smooth rationally connected variety has no pluricanonical forms. When one runs the minimal model program on a rationally connected variety $X$, the end result is a Mori fiber space. In contrast to the general type case, there may be  many different ways of expressing $X$ as birational to a Mori fiber space. Thus from the perspective of minimal model theory, one should not expect there to be a distinguished way of building a skeleton for a rationally connected $X$.

\subsection{Construction of a non-canonical skeleton} Let $X$ be a proper $K$-variety. An {\em $R$-model} of $X$ is any separated flat $R$-scheme $\mathcal{X}$ with generic fiber $\mathcal{X}_{K}\cong X$. A {\em proper snc-model} of $X$ is any proper $R$-model $\mathcal{X}$ of $X$ such that:
\vskip .2cm
\begin{itemize}
\item[{\bf (i)}]
$\mathcal{X}$ is regular;
\vskip .2cm
\item[{\bf (ii)}] The special fiber $\mathcal{X}_{k}$ is a strict normal crossing divisor in $\mathcal{X}$.
\end{itemize}
\vskip .2cm

Assume that our proper $K$-variety $X$ is regular, connected, and that it has a proper snc-model $\mathcal{X}$. As Musta\c{t}\u{a} and Nicaise explain in \cite[\S 3]{MN}, even when the special fiber $\mathcal{X}_{k}$ is not reduced, so that the $\mathfrak{m}$-adic completion $\widehat{\mathcal{X}}$ is not pluristable, we can still construct a cell complex $D(\mathcal{X}_{k})\subset\mathcal{X}^{\mathrm{an}}$ from $\mathcal{X}_{k}$ as follows.

Points in $D(\mathcal{X}_{k})$ are pairs $(\xi,x)$ where $\xi$ is any generic point of the intersection of finitely many irreducible components of $\mathcal{X}_{k}$, and where $x=(x_{1},\dots,x_{m})$ is any point in the topological space $D^{\circ}(\xi)\subset\mathbb{R}^{m}_{>0}$ cut out by the equation
$$
\underset{\mathrm{containing\ }\xi}{\underset{V_{i}\mathrm{\ of\ }\mathcal{X}_{k}}{\underset{\mathrm{components}}{\sum_{\mathrm{irreducible}}}}}\!\!\!x_{i}\ \ =\ 1.
$$
Here $m$ is the number of irreducible components of $\mathcal{X}_{k}$ containing $\xi$.
We can attach these spaces $D^{\circ}(\xi)$ to one another via boundary relations induced by the specialization relations between the generic points $\xi$ (see \cite[\S 3]{Nic} for details). Note that the resulting cell complex $D(\mathcal{X}_{k})$ coincides with the cell complex $\Delta\big(\mathrm{supp}(\mathcal{X}_{k})\big)$ that de Fernex, Koll{\'a}r, and Xu construct in \cite[\S 2]{dFKX}.

At each point $(\xi,x)$ in $D(\mathcal{X}_{k})$, with $x=(x_{1},\dots,x_{m})$, let $V_{1},\dots,V_{m}$ be the set of irreducible components of $\mathcal{X}_{k,\mathrm{red}}$ containing the generic point $\xi$. For each $1\le i\le m$, let $N_{i}$ denote the multiplicity of the component $V_{i}$ in the special fiber $\mathcal{X}_{k}$. Then the coordinatewise rescaled $m$-tuple
$$
v\ =\ \Big(\frac{x_{1}}{N_{1}},\dots,\frac{x_{m}}{N_{m}}\Big)
\ \ \ \ \ \ \mbox{in}\ \ \ \ \ \ 
\mathbb{R}^{m}_{>0}
$$
determines a multiplicative seminorm
\begin{equation}\label{divisorial seminorm}
|-|_{v}:K(X)\longrightarrow\mathbb{R}_{\ge0}
\end{equation}
defined as follows. Each rational function $f\in\widehat{\mathcal{O}}_{\mathcal{X}\!,\xi}$ admits an {\em admissible decomposition}, meaning a decomposition of the form
\begin{equation}\label{admiss decomp}
f=\sum_{u\in\mathbb{Z}^{m}_{\ge0}}a_{u}T^{u}
\ \ \ \ \ \ \mbox{with}\ \ \ \ \ \ 
a_{u}=\left\{
\begin{array}{l}
\mbox{a unit in }\widehat{\mathcal{O}}_{\mathcal{X}\!,\xi}
\\
\mbox{or }0\mbox{ otherwise}
\end{array}
\right.
\end{equation}
where $T=(T_{1},\dots,T_{m})$ is a regular system of local parameters in $\mathcal{O}_{\mathcal{X}\!,\xi}$ that give local equations for the components of $\mathcal{X}_{k}$ at $\xi$ (see \cite[Proposition 2.4.6]{MN} for details). When $f$ is in $\mathcal{O}_{\mathcal{X}\!,\xi}$ with an admissible decomposition (\ref{admiss decomp}), the seminorm $|-|_{v}$ returns
$$
|f|_{v}
\ \ =\ \ 
\underset{a_{u}\ne0}{\underset{u\in\mathbb{Z}^{m}_{\ge0}}{\mathrm{max}_{\ }}}\ \frac{1}{e^{\langle u,v\rangle}}.
$$
Here $\langle-,-\rangle:\mathbb{Z}^{m}_{\ge0}\times\mathbb{R}^{m}_{>0}\longrightarrow\mathbb{R}_{\ge0}$ denotes the standard pairing. One checks that this gives a well-defined map $|-|_{v}:\mathcal{O}_{\mathcal{X},\xi}\longrightarrow\mathbb{R}_{\ge0}$ that extends to a multiplicative seminorm on $K(X)$, and we take this to be our seminorm (\ref{divisorial seminorm}) (see \cite[\S's 2.4 \& 3.1]{MN} for details). The resulting assignment $v\mapsto|-|_{v}$ induces a closed embedding
$$
\mathrm{Sk}:D(\mathcal{X}_{k})\hookrightarrow X^{\mathrm{an}}
$$
(see \cite[\S 3.1.3 \& Proposition 3.1.4]{MN}). We refer to its image as the {\em skeleton} of $X^{\mathrm{an}}$ associated to $\mathcal{X}$, and denote it $\mathrm{Sk}(\mathcal{X})$.

\subsection{Retracting to the skeleton over characteristic-0 Laurent series}\label{skeleton building} Let $k$ be an algebraically closed field of characteristic $0$, and equip the field $K=k((t))$ of formal Laurent series with the norm $|-|:k((t))\longrightarrow\mathbb{R}_{\ge0}$ defined by
$$
\left|\sum^{\infty}_{n=m}a_{n}t^{n}\right|=\frac{1}{e^{m}}
\ \ \ \ \ \ \mbox{for}\ \ \ \ \ \ a_{m}\ne0.
$$
This norm is non-archimedean, and $k((t))$ is complete with respect to the topology that it induces. The resulting ring of integers in $k((t))$ is the ring $R=k[[t]]$ of formal power series, with maximal ideal $\mathfrak{m}=(t)$ and residue field $k=k[[t]]/(t)$.

Let $X$ be an irreducible, smooth proper variety over $k((t))$, then Nagata's embedding theorem \cite{Voj} provides us with a flat proper $k[[t]]$-model $\mathcal{X}$ of $X$ such that $\mathcal{X}$ is normal. We have no assurance, a priori, that $\mathcal{X}$ is an snc-model. However, because we work over a residue field $k$ of characteristic $0$, and because $X$ is smooth, we can use Hironaka's resolution of singularities to replace $\mathcal{X}$ with such a model if need be. After making this replacement, the construction of the previous section provides us with a skeleton
$$
\mathrm{Sk}(\mathcal{X})\ \subset\ X^{\mathrm{an}}.
$$
Nicaise and Xu prove the following homotopy equivalence result.
 
 \begin{proposition}\label{homotopyofdual}\!\!
{\bf \cite[Theorem 3.1.3]{NX}.}
There exists a continuous retraction
	$$
	\rho_{\mathcal{X}}:X^{\mathrm{an}}\longrightarrow D(\mathcal{X}_{k})
	$$
that fits into a strong deformation retraction of $X^{\mathrm{an}}$ onto $\mathrm{Sk}(\mathcal{X})$ under the identification of $D(\mathcal{X}_{k})$ with $\mathrm{Sk}(\mathcal{X})\subset X^{\mathrm{an}}$.
 \end{proposition}

\begin{remark}\label{remark: homotopy equivalence}
Let $f:X\longrightarrow Y$ be a morphism of smooth projective $k((t))$-varieties. Because $X$ and $Y$ are projective, they admit respective proper snc-models $\mathcal{X}$ and $\mathcal{Y}$ over $k[[t]]$. After finitely many blowups in the special fibers of $\mathcal{X}$ and $\mathcal{Y}$, we can assume that $f$ extends to a morphism $\mathfrak{f}:\mathcal{X}\longrightarrow\mathcal{Y}$, and that the intersection of any subset of irreducible components of the special fiber in either $\mathcal{X}$ or $\mathcal{Y}$ has a single component \cite[Remark 10]{dFKX}. The latter condition on $\mathcal{X}$ and $\mathcal{Y}$ implies that if we let $I_{\mathcal{X}}$ denote the set of irreducible components of the special fiber $\mathcal{X}_{k}$, then we can realize the dual complex $D(\mathcal{X}_{k})$ as a subcomplex of $\mathbb{R}^{I_{\mathcal{X}}}$, and similarly for $\mathcal{Y}_{k}$.  Indeed, if we let $N_i$ denote the multiplicity of the $i^{\mathrm{th}}$ component of $\mathcal{X}_{k}$, then we can realize $D(\mathcal{X}_{k})$ as the subset of the $(\# I_{\mathcal{X}}-1)$-simplex $\Delta_{I_{\mathcal{X}}}\overset{{}_{\mathrm{def}}}{=}\big\{(x_{1},\dots,x_{n})\in\mathbb{R}^{I_{\mathcal{X}}}:\sum N_i x_i =1\big\}$ consisting of exactly those faces $\Delta_{J}\subset\Delta_{I_{\mathcal{X}}}$ corresponding to subsets $J\subset I_{\mathcal{X}}$ for which $\bigcap_{j\in J}D_{j}\ne\mbox{\O}$, and likewise for $D(\mathcal{Y}_{k})$.

	In \cite[\S2]{Yu}, Yu shows that any such $\mathfrak{f}$ induces a map
	$$
	S_{\mathfrak{f}}:D(\mathcal{X}_{k})\longrightarrow D(\mathcal{Y}_{k}).
	$$
In detail, if $\mathrm{Div}_{0}(\mathcal{X})$ denotes the space of Cartier divisors on $\mathcal{X}$ supported in $\mathcal{X}_{k}$, then we have a canonical isomorphism $\mathrm{Div}_{0}(\mathcal{X})^{\ast}_{\mathbb{R}}\cong\mathbb{R}^{I_{\mathcal{X}}}$. The morphism $\mathfrak{f}:\mathcal{X}\longrightarrow\mathcal{Y}$ induces an $\mathbb{R}$-linear map
	$$
	\mathfrak{f}^{\ast}:\mathrm{Div}_{0}(\mathcal{Y})_{\mathbb{R}}\longrightarrow \mathrm{Div}_{0}(\mathcal{X})_{\mathbb{R}}.
	$$
By definition, $S_{\mathfrak{f}}:D(\mathcal{X}_{k})\longrightarrow D(\mathcal{Y}_{k})$ is the restriction of $\mathfrak{f}^{\ast}$'s dual $\mathrm{Div}_{0}(\mathcal{X})^{\ast}_{\mathbb{R}}\longrightarrow\mathrm{Div}_{0}(\mathcal{Y})^{\ast}_{\mathbb{R}}$. We can recover much of the data of $S_{\mathfrak{f}}$ from the combinatorics of the restriction of $\mathfrak{f}$ to the special fibers $\mathcal{X}_{k}$ and $\mathcal{Y}_{k}$, as the following lemma explains.

\begin{lemma}\label{combinatoricsofspecial}
Let $\mathfrak{f}:\mathcal{X}\longrightarrow\mathcal{Y}$ be a morphism of proper snc $k[[t]]$-varieties, such that the dual complex $D(\mathcal{X}_{k})$ is simplicial. For a divisor $D$, we denote the corresponding $0$-simplex of the dual complex by $\sigma_D$ .
\begin{enumerate}

\item Let $D$ be an irreducible divisor in the special fiber $\mathcal{X}_{k}$ whose image under $\mathfrak{f}$ is a divisor $E$ in $\mathcal{Y}$. Then the induced map $S_{\mathfrak{f}}$ sends $\sigma_D$ to $\sigma_E$.

\item Suppose $g: D(\mathcal{X}_{k})\longrightarrow D(\mathcal{Y}_{k})$ is a piecewise linear map such that $g(\sigma_D)=S_{\mathfrak{f}}(\sigma_D)$ for every divisor $D$. Then $S_{\mathfrak{f}}=g$.

\end{enumerate}

\end{lemma}
\begin{proof}
For the first statement, we have that for any component $E'$ of  $\mathcal{Y}_k$, either $E'=E$, or $D$ has coefficient $0$ in $\mathfrak{f}^*(E')$. Conversely,  because the coefficient of $D$ in $f^*E$ is not $0$, the dual of $\mathfrak{f}^*$ sends the dual of $D$ to the dual of $E$. 

The second statement follows by the fact that $S_{\mathfrak{f}}$ is piecewise linear.
\end{proof}

	Because the retraction maps in \cite{NX} and \cite{Yu} coincide, Yu's \cite[Proposition 2.10]{Yu} tells us that $S_{\mathfrak{f}}$ fits into a commutative diagram
	$$
	\xymatrix{
	X^{\mathrm{an}}
	\ar[d]_{f^{\mathrm{an}}}
	\ar[r]^{\rho_{\mathcal{X}}\ \ }
	&
	D(\mathcal{X}_{k})
	\ar[d]^{S_{\mathfrak{f}}}
	\\
	Y^{\mathrm{an}}
	\ar[r]_{\rho_{\mathcal{Y}}\ \ }
	&
	D(\mathcal{Y}_{k})\ .\!\!\!
	}
	$$
Thus by Proposition \ref{homotopyofdual}, to prove that $f^{\mathrm{an}}$ is a homotopy equivalence, it suffices to construct a section $D(\mathcal{Y}_{k})\!\xymatrix{{}\ar@{^{(}->}[r]&{}}\!D(\mathcal{X}_{k})$ that is a homotopy inverse to $S_{\mathfrak{f}}$.

One would like to construct this section using techniques from MMP directly on the morphism $\mathfrak{f}:\mathcal{X}\longrightarrow\mathcal{Y}$. Unfortunately, the necessary MMP techniques are not established over rings like $k[[t]]$ which are not of finite type over a field of characteristic $0$. Instead, we have to use a ``spreading out" argument in order to transfer our task to a setting where the needed MMP techniques are known.
\end{remark}


\section{Rationally Connected Varieties}\label{rcvarieties}

\subsection{Definitions and first properties} The details in this section are well known to experts; we refer to \cite{Deb} and \cite{Kol}.

One way to distinguish birational classes of varieties is by considering the sets of rational curves they contain. If $X$ is a smooth proper variety over an algebraically closed field, we say that $X$ is rationally connected if given any two closed points $p$ and $q$ of $X$, there is a map $f: \mathbb{P}^1 \longrightarrow X$ such that both $p$ and $q$ are in the image of $f$.  Projective $n$-space is rationally connected for instance, since any two closed points of $\mathbb{P}^n$ are contained in a line. Many other familiar varieties are also rationally connected, including Grassmanians, toric varieties, and smooth Fano varieties.

Because we will be working over $k((t))$, which is not algebraically closed, we will need a slightly more general definition than the one above.

\begin{definition}\label{rationally connected def} Assume that $F$ is a field of characteristic $0$. A proper, geometrically irreducible scheme $X$ over $F$ is {\em rationally connected} if there exists a variety $M$ over $F$, and a morphism $u: M \times\mathbb{P}^1 \longrightarrow X$ such that the induced map $u^{(2)}: M \times\mathbb{P}^1\times\mathbb{P}^1 \longrightarrow X \times_F X$ is dominant.
\end{definition}

This definition follows \cite[Definition 4.3]{Deb}, and corresponds to the definition of a \emph{separably rationally connected} scheme in \cite[Definition IV.3.2]{Kol}. The two notions are equivalent in characteristic $0$. Note that in Definition \ref{rationally connected def}, it is equivalent to take $u$ to be a rational map instead of a morphism. When the field $F$ is algebraically closed and $X$ is smooth, this definition is equivalent to the condition that every pair of closed points of $X$ be connected by a rational curve.

Definition \ref{rationally connected def} behaves well under birational modification: if $X'$ is birational to $X$, then $X'$ is rationally connected if and only if $X$ is. Thus all rational varieties are rationally connected. Rationally connected varieties also enjoy the additional property that the image of a rationally connected variety is rationally connected:

\begin{proposition}
Let $X$ be a rationally connected scheme over $F$, and let $f: X \longrightarrow Z$ be a surjective morphism of proper $F$-schemes. Then $Z$ is rationally connected.

\end{proposition}
\begin{proof}
We have $u: M \times \mathbb{P}^1 \longrightarrow X$, with $u^{(2)}$ dominant. Compose to get a morphism $f\circ u: M\times \mathbb{P}^1 \longrightarrow Z$, and observe that surjectivity of $f$ implies that the induced map $M \times \mathbb{P}^1 \times \mathbb{P}^1 \longrightarrow Z \times Z$ is dominant.
\end{proof}
\subsection{Rational connectivity in families} Rational connectivity behaves well in families. If $\pi:Y\longrightarrow S$ is a proper, smooth map of varieties over an algebraically closed field $F$ of characteristic $0$, then the locus of points $s\in S$ with rationally connected fiber $Y_{s}$ is both open in $S$ and is a countable union of closed subsets of $S$ \cite[Corollary IV.3.5.2  \& Theorem IV.3.11]{Kol}. In particular, every fiber $Y_{s}$ is rationally connected as soon as one of these fibers is.

When $F$ is algebraically closed of characteristic $0$, we say that a morphism $\pi:Y\longrightarrow S$ of $F$-varieties is a {\em rationally connected fibration} if $\pi$ is proper, and if the fiber of $\pi$ over a general point of $S$ is rationally connected. This implies that every fiber contained in the smooth locus of $\pi$ is rationally connected. From the above discussion, we immediately have:

\begin{proposition}\label{rcfib}
Let $S$ be a variety with generic point $\nu$, and let $\pi:Y\longrightarrow S$ be a proper map. If the fiber $Y_\nu$ is rationally connected, then $\pi$ is a rationally connected fibration.
\end{proposition}



\subsection{Rational connectivity and the Minimal Model program} Loosely speaking, the prevalence of rational curves on a variety measures the failure of its canonical divisor to be positive, and so being rationally connected is a strict condition on the birational geometry of a variety $X$.  For example, if $X$ is a smooth Fano variety ($-K_X$ is ample) then we expect to find many rational curves on $X$, and in fact such an $X$ is always rationally connected \cite{Cam, KMM}. While the converse is not true, in a certain sense rationally connected varieties can be built out of Fano varieties using the minimal model program.

The minimal model program is a program for classifying algebraic varieties up to birational equivalence \cite{KM}. Given a smooth projective variety $X$, the main goal of the MMP is to produce a variety $\widetilde{X}$ birational to $X$, with at worst mild singularities, such that the canonical divisor $K_{\!\widetilde{X}}$ is nef (that is, nonnegative on every curve). One calls the variety $\widetilde{X}$ a \emph{minimal model} of $X$.

The strategy for producing a minimal model of $X$ is to systematically eliminate curves $C$ in $X$ satisfying $K_X\cdot C <0$. This is accomplished by either contracting a divisor $D$ in $X$ covered by such curves, or by applying a birational operation called a flip. Assuming that the procedure of repeatedly applying these two operations terminates after finitely many steps, the end result is either a minimal model or a map $\widetilde{X} \longrightarrow Z$ of relative Picard number $1$, such that $-K_{\!\widetilde{X}}$ is relatively ample. The map $\widetilde{X} \longrightarrow Z$ is called a \emph{Mori fiber space}. In general, termination of flips is a difficult open problem in birational geometry. However, Birkar, Cascini, Hacon, and McKernan \cite{BCHM} established the existence of minimal models for varieties of general type by showing termination of flips for a special MMP called the {\em MMP with scaling}. 

The operations constituting the MMP are closely related to the existence of rational curves on the variety $X$: if there exists a curve $C$ in $X$ such that $K_X \cdot C<0$, then there is always a rational curve with this same property. Whenever we contract a divisor in running the MMP, we do so by contracting a family of rational curves that cover the divisor.

Having a covering family of rational curves on a smooth variety $X$ precludes the existence of a section of $nK_X$ for any positive $n$. In fact slightly more is true. Recall that $N^{1}(X)_{\mathbb{Q}}$ denotes the group of $\mathbb{Q}$-Cartier divisors on $X$ modulo numerical equivalence, and that $N^{1}(X)_{\mathbb{Q}}$ is a finite dimensional $\mathbb{Q}$-vector space. A $\mathbb{Q}$-Cartier divisor $D$ on $X$ is {\em pseudoeffective} if the class of $D$ in $N^{1}(X)_{\mathbb{R}}=N^{1}(X)_{\mathbb{Q}}\otimes\mathbb{R}$ lies inside the closure of the cone spanned by effective $\mathbb{Q}$-Cartier divisors.

\begin{proposition}\label{notpsef}
If $X$ is a smooth and rationally connected projective positive dimensional variety over an algebraically closed field, then $K_X$ is not pseudoeffective.
\end{proposition}
Note that this is true under the weaker hypothesis that $X$ is uniruled.
\begin{proof}
By \cite[Ex 4.7, Cor 4.11]{Deb}, there is a family of rational curves on $X$ such that 
\begin{enumerate}
\item $-K_X\cdot C \leq -2 $ for any curve $C$ in the family.

\item For a general point $p \in X$, there is a curve $C$ in the family passing through $p$.

Choose any ample divisor $A$ on $X$. Then for sufficiently small $\lambda >0$, $\lambda A \cdot C <2$. Thus $(K_X+\lambda A)\cdot C<0$. Since given any closed subset of $X$ we can choose $C$ to not be contained in that set, we must have that no multiple of $(K_X+\lambda A) $ has a section.
\end{enumerate}
\vskip -.6cm
\end{proof}

As a consequence, by \cite[Cor 1.3.3]{BCHM}, running any MMP with scaling on a smooth rationally connected $X$ produces a Mori fiber space $\tilde{X} \longrightarrow  Z$. The base $Z$ is the image of a rationally connected variety, so it too is rationally connected. If we repeatedly run an MMP with scaling starting with a rationally connected variety, we produce a long sequence of divisorial contractions, flips, and Mori fibrations terminating at a single point.


\section{Approximation by curve models}\label{spreading out section}

\subsection{Rationally connected spreading out} Throughout the remainder of this article, we fix an algebraically closed field $k$ of characteristic $0$, and let $K=k((t))$ with ring of integers $R=k[[t]]$.

The present section is devoted to a proof that if $f:X\longrightarrow Y$ is a surjective morphism of smooth projective $k((t))$-varieties, then we can approximate any snc $k[[t]]$-model $\mathcal{X}\longrightarrow\mathcal{Y}$ of this morphism by a morphism over a smooth curve. Moreover, we can do so in a way that preserves the morphism's general relative smoothness and relative rational connectivity when these properties hold over general points of $Y$. While this follows from the fact that rational connectedness is a geometric property, we include a proof for completeness.

\begin{lemma}\label{basechange}
Let $X$ be a geometrically irreducible proper variety over a field $L$ of characteristic $0$. Let $L'$ be a field extension of $L$. Then $X$ is rationally connected only if $X_{L'}$ is.
\end{lemma}

In fact the converse also holds but we will not need this fact.
\begin{proof}
Assume $X_{L'}$ is rationally connected. Then we have a variety $M$ over $L'$ with a map $f:M \times \mathbb{P}_{L'}^1 \longrightarrow X_{L'}$ such that the induced map
$$
f^{(2)}\ \!:\ M\times\mathbb{P}^{1}\times\mathbb{P}^{1}\xrightarrow{\ \ \ \ \ }X\underset{\!\!\!L'\!\!\!}{\times}X
$$
is dominant. Let $G$ be the graph of $f$. By spreading out \cite[\S 8 \& Th\'eor\`eme 11.6.1]{Groth} we can find a finitely generated $L$-subalgebra $A$ of $L'$, along with varieties $M_A$ and $G_A$ over $A$, such that $G_A\subset M_A\times \mathbb{P}^1 \times X_A$ and both $M$ and $G$ are given by basechange to $L'$. 

We now want to show that $G$ is the graph of a rational map from $M_A \times \mathbb{P}^1$ to $X_A$.

Note that we have a commutative diagram of function fields:
$$
\xymatrix{
k(G_A)\ar[r] & k(G) 
\\ k(M_A\times\mathbb{P}^1)\ar[r]\ar[u] & k(M\times\mathbb{P}^1) \ar[u]
\\ k(A) \ar[r] \ar[u] & L'\ar[u]
}
$$
The bottom square is co-Cartesian, as is the large square. Hence the top square is also co-Cartesian. Since $k(G)$ has degree $1$ over $k(M\times\mathbb{P}^1)$, this implies that $k(G_A)$ has degree $1$ over $k(M_A\times\mathbb{P}^1)$. This gives us an open subset $U \subset M_A\times\mathbb{P}^1$ equipped with a morphism $\phi:U\longrightarrow X$. Note that $\phi$ induces a morphism
$$
\phi^{(2)}\ \!:\ U\underset{M_A}{\times}U\xrightarrow{\ \ \ \ \ \ }X_A\underset{\! A\!}{\times}X_A.
$$
Because $f^{(2)}:M\times\mathbb{P}^{1}\times\mathbb{P}^{1}\longrightarrow X\times_{L'} X$ is dominant, this morphism $\phi^{(2)}$ is also dominant.
By resolving singularities of $M_A$, we can choose $M'_A$ so that the rational map
$$
\phi'\ \!:\ M'_A \times\mathbb{P}^{1}\dashrightarrow X_A
$$
induced by $\phi$ is defined away from a codimension-$2$ subset.  Let $M''_A$ be the complement of the image of this set in $M'_A$.  Thus we have a map $\phi':M''_\eta\times\mathbb{P}^{1}\longrightarrow X_A$ such that the induced map
$$
\phi'^{\ \!(2)}\ \!:\ M''_A \times\mathbb{P}^{1}\times\mathbb{P}^{1}
\xrightarrow{\ \ \ \ \ \ \ }
X_A\underset{\!A\!}{\times}X_A
$$
is dominant, since this map agrees with $\phi^{(2)}$ on an open set. Composing $\phi$ with the base change $X_A \rightarrow X$ gives a map 
$$
\zeta\ \!:\ M''_A \times\mathbb{P}^{1}\longrightarrow X
$$
and the induced map $\zeta^{(2)}$ is dominant. Since $M''_A$ is finite type over the ground field $k$, $X$ is rationally connected.
\end{proof}

\begin{proposition}\label{curvemodel}
Let $f:X\longrightarrow Y$ be a surjective morphism of smooth projective varieties over $k((t))$. Assume that for a general geometric point $q$ of $Y$, the fiber $X_q$ is smooth and rationally connected. If $\mathfrak{f}:\mathcal{X}\longrightarrow\mathcal{Y}$ is a map of snc $k[[t]]$-models that extends $f$, then for any positive integer $n$ there exists a pointed affine curve $C\ni p$ over $k$, along with a map $g_C:V_{C}\longrightarrow W_{C}$ of regular $C$-varieties satisfying:

\begin{enumerate}
\item The special fibers $V_p$ and $W_p$ are snc inside $V_C$ and $W_C$.

\item There are isomorphisms $\nu_1: (\mathcal{X}_k)_n \to (V_p)_n$ and $\nu_2: ( \mathcal{Y}_k)_n \to (W_p)_n$ on the $n$th order neighborhoods inducing the following commutative diagram:

\begin{equation}
\xymatrix{
(\mathcal{X}_k)_n\ar[r] \ar[d]& (V_p)_n\ar[d]
\\ (\mathcal{Y}_k)_n\ar[r]& (W_p)_n 
}
\end{equation}

\item The fiber of $g_C$ over a general point $q$ of $W_p$ is rationally connected.
\end{enumerate}
\end{proposition}
\begin{proof}

As above, we can apply \cite[\S 8 \& Th\'eor\`eme 11.6.1]{Groth} to choose an affine $k$-variety $S$, along with varieties $V$, $W$, and a $\Theta \subset V \times_{S} W$ such that $V$, $W$, and $\Theta$ pull back to $\mathcal{X}$, $\mathcal{Y}$, and the graph $\Gamma$, respectively. We may assume $\Theta$ is isomorphic to $V$ over $S$, so that $\Theta$ gives the graph of a morphism $g:V \longrightarrow W$. Let $A$ denote the coordinate ring of $S$.

The construction of the curve $C$ and flat families $V_{C}\longrightarrow C$ and $W_{C}\longrightarrow C$ proceeds exactly as in the proof of \cite[Proposition 5.1.2]{MN}: Given a positive integer $n$, we use Greenberg approximation \cite{Green} to find a morphism
\begin{equation}\label{morphism to Henselization}
A\longrightarrow\mathcal{O}^{\ \!\mathrm{h}}_{\!\mathbb{A}^{\!1}_{k},(t)}
\end{equation}
to the Henselization of the stalk at the origin in $\mathbb{A}^{\!1}_{k}$, such that the composite of (\ref{morphism to Henselization}) with the projection
$$
\widehat{\mathcal{O}^{\ \!\mathrm{h}}_{\!\mathbb{A}^{\!1}_{k},(t)}}\ \cong\ k[[t]]
\longrightarrow
k[t]/(t^{n})
$$
factors through a morphism $A\longrightarrow k[[t]]/(t^{n})$. Because the Hensilization is the stalk for the \'etale topology, we can find a smooth curve $C$ and a morphism
\begin{equation}\label{curve morphism}
C\longrightarrow\mathrm{Spec}_{\ \!}A
\end{equation}
through which the morphism dual to (\ref{morphism to Henselization}) factors. If we define $V_{C}$ and $W_{C}$ to be the pullbacks of $V$ and $W$ along this morphism (\ref{curve morphism}), then it follows that
\[
\mathrm{Spec}_{\ \!}k[t]/(t^n)\underset{C}{\times}V_{C}
\ \ \cong\ \ 
\mathrm{Spec}_{\ \!}k[t]/(t^n)\underset{\!\!k[[t]]\!\!}{\times}\mathcal{X}
\ \ \cong\ \ 
\mathrm{Spec}_{\ \!}k[t]/(t^n)\underset{S}{\times}V
\]
and similarly for the triple $W_{C}$, $\mathcal{Y}$, and $W$, as well as for $\Theta_C, \Gamma,$ and $\Theta$. We have now established (ii).
Next, consider the flat locus of $g: V \longrightarrow W$. This is an open set, so its image in $S$ is open. Let $T \subset S$ be the complement. Now, because the map $\mathrm{Spec}_{\ \!}k[[t]] \longrightarrow S$ is dominant, it does not factor through $T$. Thus there is some $n \geq 2$ such that the induced map from $\mathrm{Spec}_{\ \!}k[t]/(t^n)$ does not factor through $T$. Thus we may assume that  $C$ is not contained in $T$.

The requirement that $\mathcal{X}$ and $\mathcal{Y}$ be snc over $k[[t]]$ includes the requirement that they are both regular. Because the Zariski tangent space at points in the special fiber depends only on the fiber over $\mathrm{Spec}_{\ \!}k[t]/(t^2)$, this implies that $V_{C}$ and $W_C$ are both smooth in a neighborhood of their respective special fibers.  Localizing if necessary, we may therefore choose $C$ so that $V_C$ and $W_C$ are smooth \cite[Proposition 5.1.2.(6)]{MN}.  In particular, this implies that a general fiber of the morphism from $V \longrightarrow W$ is smooth.

We have that the generic fiber of the morphism $X \longrightarrow Y$ is rationally connected, but this is the base change of the generic fiber of $g: V \longrightarrow W$, so by Lemma \ref{basechange} the generic fiber of $g$ is rationally connected. 

Note that irreducibility of $Y$ implies that $W$ is irreducible.

Because rational connectedness is both open and closed in smooth families, we have that for closed points in the flat locus of $g$, the fiber of $g$ is rationally connected whenever it is smooth. Since both $V_C$ and $W_C$ are smooth, and $C$ is not contained in $T$, this means that over a general point of $W_C$, the fiber of $g_C$ is rationally connected.
\end{proof}


\section{Homotopy type of the dual complex}\label{applying dFKX}

\subsection{Singularities of the MMP} De Fernex, Koll\'{a}r, and Xu have shown recently that running the MMP on certain varieties with snc divisors preserves the homotopy type of the dual complex  \cite{dFKX}. As a result they are able to show the following
\begin{theorem}\cite[Thm 1]{dFKX}
Let $X$ be an isolated log terminal singularity. Then for any log resolution, the dual complex of the exceptional divisor is contractible.
\end{theorem}
\noindent
Applying their techniques to the case of a degeneration of rationally connected varieties, they are able to show

\begin{theorem}\cite[Thm 4]{dFKX}\label{degen}
Let $X\longrightarrow C$ be a degeneration of rationally connected varieties over a curve, with special fiber $X_0$. Assume $X$ is smooth and that $X_0$ is snc. Then the dual complex of $X_0$ is contractible.
\end{theorem}

Log terminal singularities are a natural class of singularities arising in the MMP. Examples of log terminal singularities are provided by quotient singularities and Gorenstein rational singularities. We will need to introduce some related notions for pairs $(X,\Delta)$ so that we may apply the techniques of \cite{dFKX}. The standard reference is \cite{KM}.

Let $X$ be a normal variety over an algebraically closed field of characteristic $0$, and let $\Delta$ be an effective $\mathbb{Q}$-divisor on $X$ such that $K_X+\Delta$ is Cartier. Let $f:Y \longrightarrow X$ be a {\em log resolution} of $(X,\Delta)$. This means that:
\vskip .2cm
\begin{itemize}
\item[{\bf (i)}] The variety $Y$ is smooth;
\item[{\bf (ii)}] The map $f$ is proper and birational;
\item[{\bf (iii)}] If we let $\{E_i\}$ be the set whose elements are all components of the exceptional divisors of $f$ and all components of the strict transform of $\Delta$, then the $E_i$ form a simple normal crossing divisor.
\end{itemize}
\vskip .2cm
 
We can compare the divisor $K_X+\Delta$ on $X$ with its counterpart $K_Y+\sum E_i$ in $Y$. These must agree away from the $E_i$, and so we have
$$
K_Y+\sum E_i\ =\ f^*(K_X+\Delta)+\sum a_i E_i
$$
where the pullback of $K_X+\Delta$ is a well defined $\mathbb{Q}$-divisor, and where the $a_i$ are rational numbers, called the {\em log discrepancies}. If every one of the log discrepancies satisfies $a_i \geq 0$, then we say that $(X, \Delta)$ is \emph{log canonical}, or \emph{lc}. If each of the log discrepancies satisfies $a_i >0$, then we say that $(X, \Delta)$ is \emph{Kawamata log terminal}, or \emph{klt}. We say that $X$ itself is \emph{log terminal} if the pair $(X,0)$ is klt.

A subvariety $V \subset X$ is called a {\em log canonical center} if it is of the form $f(E_i)$ for some $E_i$ with vanishing log discrepancy: $a_i=0$.  The pair $(X,\Delta)$ is called \emph{divisorial log terminal}, or \emph{dlt}, if $(X,\Delta)$ is log canonical and for every log canonical center $V$ of $(X,\Delta)$, there is a neighborhood of the generic point of $V$ in which $(X,\Delta)$ becomes snc. The related notion of \emph{quotient-dlt}, or \emph{qdlt}, is less well known, so we refer to \cite[Section 5]{dFKX}. The pair $(X,\Delta)$ is qdlt if every codimension-$d$, log canonical center of $(X,\Delta)$ is locally the intersection of $d$ distinct components of $\Delta$, each with coefficient $1$. The name ``quotient dlt'' comes from the fact that a qdlt singularity is characterized by being locally a quotient of a dlt singularity by an abelian group \cite[Prop 34]{dFKX}.

These singularity classes obey the following logical dependencies: 
$$
\text{klt}  \Rightarrow \text{dlt} \Rightarrow \text{qdlt} \Rightarrow \text{lc}
$$

\begin{example}
Let $X=\mathbb{A}^2$ and let $\Delta$ be the coordinate axes both with coefficient $1$. Then $(X,\Delta)$ is dlt, qdlt, and lc, but not klt, and the log canonical centers are $V(x)$, $V(y)$, and $V(x,y)$. Let $X'$ be the affine quadric cone and $\Delta'$ the union of two lines on this cone. Then $(X',\Delta')$ is the quotient of $(X,\Delta)$ by a $\mathbb{Z}/2$ action which preserves the components of $\Delta$, so $(X',\Delta')$ is qdlt and lc. However, $(X',\Delta')$ is not dlt (nor klt) because it is not snc in a neighborhood of the vertex of the cone, which is a log canonical center.
\end{example}

The first important property of these types of singularities is that each class is preserved under running a $K_X+\Delta$ MMP. More precisely, if $g: X \dashrightarrow X'$ is a divisorial contraction or a flip for the pair $(X,\Delta)$, which is klt/dlt/qdlt/lc, then the pair $(X', g_* \Delta)$ is also klt/dlt/qdlt/lc. If $(X,\Delta)$ is qdlt we define the dual complex $D(\Delta^{=1})$ to be the intersection complex of the divisors that have coefficient $1$ in $\Delta$. If every divisor of $\Delta$ has coefficient $1$, we just write $D(\Delta)$ for brevity.

\subsection{Proof of the Main Theorem} Before giving a proof of Theorem \ref{main}, we need to establish an analogous result for a morphism of varieties over a curve. Say that a projective qdlt pair $(X,\Delta)$ is \emph{tidy} if every component of $\Delta$ has coefficient $1$, and if the dual complex is a simplicial complex. The latter means that the intersection of any subset of components of $\Delta$ is either empty or irreducible. In the sequel, if $(X,\Delta_X)$ and $(Y,\Delta_Y)$ are tidy projective qdlt pairs, and $f:X \dashrightarrow Y$ is a rational map such that $f_*\Delta_X=\Delta_Y$, then we say that a morphism of simplicial complexes $\alpha: D(\Delta_Y) \longrightarrow  D(\Delta_X)$ is a \emph{combinatorial section} for $f$ if $\alpha$ is an inclusion and if for every component $E$ in $\Delta_Y$, we have $f(\zeta_{\alpha(E)})=\zeta_E$, where $\zeta_E$ and $\zeta_{\alpha(E)}$ are the generic points of $E$ and $\alpha(E)$, respectively. 

\begin{theorem}\label{mapdegen}
Let $C$ be an affine curve over an algebraically closed field $k$ of characteristic $0$, and let $p$ be a $k$-valued point of $C$. Let $X$ and $Y$ be projective $C$-varieties
$$
\pi_X: X \longrightarrow C
\ \ \ \ \ \ \mbox{and}\ \ \ \ \ \ 
\pi_Y: Y \longrightarrow C
$$
and let $\Delta_X$ and $\Delta_Y$ be the respective reduced fibers of $X$ and $Y$ over $p$. Assume that $(X,\Delta_X)$ and $(Y,\Delta_Y)$ are both tidy qdlt pairs, and let $f:X \longrightarrow Y$ be a surjective morphism over $C$ such that the general fiber $X_q$ is rationally connected, and such that either $f$ is birational or $K_{X/Y}$ is not pseudoeffective. 

Then there exists a tidy $\mathbb{Q}$-factorial qdlt pair $(Y',\Delta_{Y'})$ such that the following holds: There is a birational map $f':X \dashrightarrow Y'$ that has a combinatorial section $\mu: D(Y') \to D(X)$ that is a homotopy equivalence, there is a birational morphism $\psi:Y'\to Y$ which is an isomorphism over the special fiber, and the composition of these maps agrees with $f$ away from the special fiber. In fact $(Y',\Delta_{Y'})$ is a minimal model of $(Y,\Delta_Y)$.\end{theorem}

The proof of this theorem is adapted directly from \cite{dFKX}. The key point is that under these hypotheses, each birational step of the MMP and each Mori fiber space induces such a homotopy equivalence. Before starting the proof, we need a few lemmas to ensure that in the positive relative dimension case we always have that $K_{X/Y}$ is not pseudoeffective. This follows automatically by Proposition \ref{notpsef} when $X_q$ is smooth, or even terminal, but these properties are not preserved for the base of a Mori fiber space. Thus we need to take a partial resolution preserving the dual complex.

\begin{lemma}\label{exceptionaldivs}
Let $(X,\Delta)$ be a qdlt pair. Suppose $f:Z\longrightarrow X$ is a birational morphism, and $E$ is an exceptional divisor on $Z$ with log discrepancy less than or equal to $1$ relative to the pair $(X,\Delta)$. Then either $f(E)$ is contained in a component of $\Delta$, or $f(E)$ contains no log canonical center of $(X,\Delta)$.

\end{lemma}

\begin{proof}
Suppose $f(E)$ contains a log canonical center $W$ of $(X,\Delta)$, but that $f(E)$ is not contained in any component of $\Delta$.  Since $(X,\Delta)$ is qdlt there is a neighborhood $U$ of $W$ in $X$ such that $U$ has a finite cover by an snc pair $(U', \Delta')$ branched only along the components of $\Delta_{U}$. Since $f(E)$ has log discrepancy less than or equal to $1$ and is not contained in any component of $\Delta$, the scheme $U$ must be singular along $f(E)$. This is a contradictions the fact that $U'$ is smooth along the preimage of $f(E)$, which lies outside of the branch locus.
\end{proof}

\begin{lemma}\label{partialresolution}
Let $f:X \longrightarrow  Y$ be a morphism of projective varieties over a pointed curve $p \in C$, such that the general fiber of $f$ is positive dimensional and rationally connected. Let $\Delta$ be the reduced special fiber of $\pi: X \longrightarrow  C$. Suppose that $X$ is $\mathbb{Q}$-factorial and that the pair $(X,\Delta)$ is qdlt. Then there exists a pair $(X',\Delta')$, along with a birational morphism $g:X' \longrightarrow  X$, such that $\Delta'$ is the reduced special fiber of $X'$ over $C$, $K_{X'/Y}$ is not pseudoeffective, and $g$ is an isomorphism in the neighborhood of the generic point of any log canonical center of $(X',\Delta')$. In particular, $(X',\Delta')$ is qdlt and $g_*$ induces an isomorphism from the dual complex of $(X',\Delta')$ to that of $(X,\Delta)$, with inverse $g^{-1}_*$, and $(X',\Delta')$ is tidy if $(X,\Delta)$ is.
\end{lemma}

\begin{proof}
Choose a log resolution $h:Z\longrightarrow  X$ of the pair $(X,\Delta)$. We have that 

\[
K_Z + \Lambda = f^*(K_X+\Delta)
\]
where all coefficients of $\Lambda$ are at most $1$ (possibly negative) and every divisor in the support of $\Lambda$ is either a strict transform of a divisor in $\Delta$ or is exceptional. As $h:Z \longrightarrow  X$ is a log resolution, the support of $\Lambda$ is snc. Let $\tilde{\Lambda}$ be the sum of terms of $\Lambda$ that lie in the special fiber. Then we may choose $\Lambda'$ supported on the special fiber such that $\Lambda' \geq \tilde{\Lambda}$, such that $\Lambda'$ is dlt and every coefficient of $\Lambda'$ is positive. In particular, the strict transform of each component of $\Delta$ has coefficient $1$ in $\Lambda'$.
 
For small $\varepsilon>0$, the pair $(Z,\Lambda'-\varepsilon h^*(\pi^*(p)))$ is klt. If we let $\Theta$ be the sum of the exceptional divisors contained in the special fiber, then for sufficiently small $\delta>0$, the pair $(Z,\Gamma)=(Z,\Lambda'-\varepsilon h^*(\pi^*(p))+\delta \Theta)$ is also klt.

Run an MMP with scaling for the pair $(Z,\Gamma)$ over $(X,\Delta-\varepsilon \pi^*(p))$ to get a birational map $\alpha: Z \dashrightarrow X'$ and a birational morphism $g:X' \longrightarrow  X$, with $X'$ normal and $\mathbb{Q}$-factorial. The divisor 
	$$
	-E\ =\ (K_{X'}+\alpha_*\Gamma)-g^*\big(K_X+\Delta-\varepsilon \pi^*(p)\big)
	$$
is exceptional and $g$-nef, so by the negativity lemma \cite[Lemma 3.39]{KM}, $E$ is effective. Since $X$ is normal and $\mathbb{Q}$-factorial, the exceptional locus of $g$ has pure codimension $1$. This lets us write
\[
g^*(K_X+\Delta)\ =\ K_{X'}+\alpha_*\Gamma+\varepsilon\ \!g^*\big(\pi^*(p)\big)+E\ =\ K_{X'}+\alpha_*(\Lambda'+\delta\Theta) + E .
\]

Since $(X,\Delta)$ is log canonical, the pair $(X', \alpha_*\Gamma+\varepsilon g^*(\pi^*(p))+E)$ is also log canonical. Let $F$ be an exceptional divisor of $h:Z\longrightarrow  X$. We wish to know when $F$ is contracted by $\alpha$. 

Suppose $F$ is not contracted by $\alpha$. Then the coefficient of $F$ in $K_Z+\Lambda'+\delta\Theta-f^*(K_X+\Delta)$ is the coefficient of $\alpha_* F$ in $-E$, thus cannot be positive. In particular, every component of $\Theta$ is contracted by $\alpha$. Let $\Delta'=\alpha_*(\Lambda'+\delta\Theta)$, and then $\Delta'$ consists of all components of the special fiber of $X'$ with multiplicity $1$. Since $(X', \Delta'+E)$ is log canonical, so is $(X',\Delta')$. 

Consider an exceptional divisor $G$ of $g$. Since $K_X'+\Delta'+E=g^*(K_X+\Delta)$, with $E$ effective, we have that $G$ has log discrepancy at most $1$ relative to the pair $(X,\Delta)$. We also know that the divisor $G$ does not lie in the special fiber, because its strict transform in $Z$ would be a component of $\Theta$, and thus contracted by $\alpha$. Hence the image $g(G)$ is not contained in any log canonical center of $(X,\Delta)$, and by Lemma \ref{exceptionaldivs}, $g(G)$ contains no log canonical center of $(X,\Delta)$. This means that the exceptional divisors of $g$ cannot contain any log canonical centers of $(X',\Delta')$, and therefore $g$ is an isomorphism in a neighborhood of any generic point of any  log canonical center of $(X',\Delta')$. Likewise, $g^{-1}$ is an isomorphism in a neighborhood of the generic point of any log canonical center of $(X,\Delta)$. Hence the map $g_*$ induces an isomorphism from the dual complex of $(X',\Delta')$ to that of $(X,\Delta)$, and $g^{-1}_*$ induces its inverse.

Let $(\tilde{X}, \tilde{\Delta})$ be a log resolution $\beta: \tilde{X} \longrightarrow  X'$. Write
\[
K_{\tilde{X}/Y}=K_{X'/Y}+a_i E_i + b_j F_j,
\]
where the $E_i$ are supported away from the special fiber and the $F_j$ are supported on the special fiber. We have that each $a_i\geq 0$, and each $b_j \geq -1$. Choose $m>0$ so that $c_j F_j=m\alpha^*(\pi^*(p))$  satisfies $c_j>-b_j$ for each $i$.

Since the variety $\tilde{X}$ is smooth and a general fiber $\tilde{X}_q$ over $Y$ is rationally connected, we have by Prop \ref{notpsef} that $K_{\tilde{X}/Y}$ is not pseudoeffective. But
\[
K_{\tilde{X}/Y}=K_{X'/Y}+a_i E_i +(c_j+b_j)F_j.
\]
So $K_{X'/Y}$ is not pseudoeffective either.
\end{proof}

\begin{proof}[Proof of Thm \ref{mapdegen}]
Choose a small $\varepsilon>0$, and consider the pair $\big(\ \!X,\ \!\Delta_X- \varepsilon \pi^{-1}(p)\ \!\big)$. This is a klt pair, so by \cite{BCHM} we can run MMP with scaling for this pair over $Y$. Because $\varepsilon \pi^{-1}(p)$ is numerically trivial, the MMP with scaling is also a $K_X+\Delta_X$ MMP, and it preserves the property of being qdlt. We can always choose our MMP to also preserve the property of being $\mathbb{Q}$-factorial. Now either $X$ and $Y$ have the same dimension, and the MMP terminates at $W=Y'$, or $K_{X_q}$ is not pseudoeffective and the MMP must terminate in a Mori fiber space $h:W \longrightarrow  Z$, and since $X_q$ is rationally connected, so is $Z_q$. In either case, the birational map $\beta: X \dashrightarrow W$ is a composition of flips and divisorial contractions.

Set $\Delta_W=\beta_*\Delta_X$. By \cite[Cor 22]{dFKX}, we know that $\beta^{-1}_*$ identifies $\Delta_W$ with a sub-complex of $\Delta_X$ and that $\Delta_X$ collapses onto $\Delta_W$. Thus $(W,\Delta_W)$ is tidy and the inclusion induced by $\beta^{-1}_*$ is a combinatorial section of $\beta$ and also a homotopy equivalence. In addition, the divisor $\Delta_W$ is the reduced special fiber of the morphism $W \longrightarrow C$. If $W$ is birational to $Y$ then take $(Y',\Delta_{Y'}) = (W,\Delta_W)$ and then the fact that $(W,\Delta_W)$ is minimal over $Y$ implies that the only exceptional divisors for the birational morphism $\psi:Y' \to Y$ are components of the special fiber.

Otherwise we have a Mori fiber space $h:W \longrightarrow  Z$. Let $\Delta_Z$ be the reduced special fiber of $Z \longrightarrow  C$. By \cite[Prop 40]{dFKX}, the pair $(Z,\Delta_Z)$ is qdlt, and push forward by $h$ induces an isomorphism between $D(\Delta_Z)$ and $D(\Delta_W)$. Thus the inverse is a combinatorial section for $h$, and a homotopy equivalence. Now either $Z$ is birational to $Y$, or using Lemma \ref{partialresolution}, we produce a birational map $g:Z' \longrightarrow  Z$, such that if $\Delta_{Z'}$ is the reduced special fiber of $Z'$, then $(Z',\Delta_{Z'})$ is qdlt and $g_{*}$ induces an isomorphism between the dual complexes $D(\Delta_Z)$ and $D(\Delta_{Z'})$. Hence $Z'$ is tidy and the inverse of $g_{\ast}$ is a combinatorial section for $g$ and a homotopy equivalence. Moreover, the fiber $Z'_q$ over a general point of $Y$ is rationally connected, and $K_{Z'/Y}$ is not pseudoeffective.

Consider the rational map $g \circ h\circ\beta: X \dashrightarrow Z'$. Since $g$, $h$ and $\beta$ have combinatorial sections that are also homotopy equivalences, so does $g \circ h\circ\beta$. Since $Z'$ has strictly smaller dimension than $X$, we may assume inductively that the map $\gamma: Z' \longrightarrow  Y'$ has a combinatorial section that is a homotopy equivalence. Thus so does $f'$.
\end{proof}

\begin{remark}
Let $(X,\Delta_X)$ be an snc degeneration of smooth rationally connected varieties over a pointed curve $p \in C$. Then we may apply Thm \ref{mapdegen} with $(Y,\Delta_Y)=(C,p)$. The combinatorial section distinguishes a component $D$ of $\Delta_X$, which is birational to a tower of Mori fiber spaces. Thus, by \cite{GHS}, $D$ is rationally connected. By running a Galois equivariant MMP, we may relax the requirement that $k$ is algebraically closed, as long as the $k$-irreducible components of the special fiber have simple normal crossings. This gives an affirmative answer to \cite[Question 11]{Kol2} in the simple normal crossings case. See also \cite{HX}.
We are very grateful to Chenyang Xu for pointing this out.

\end{remark}

To complete the analogue of Theorem \ref{main} over a curve, we must compare the dual complexes of $(Y,\Delta_Y)$ and $(Y', \Delta_{Y'})$. These are PL homeomorphic \cite[Prop 11]{dFKX}, but we will need to specify the homeomorphism.

Note that because $Y$ and $Y'$ are tidy, the morphism $\psi:Y'\longrightarrow Y$ provided by Theorem \ref{mapdegen} gives us a map $S_{\psi}: D(\Delta_{Y'}) \to D(\Delta_Y)$ induced by the dual to the pullback $\psi^*$ as in Remark \ref{remark: homotopy equivalence}.

\begin{lemma}\label{PLhomeo}
The map $S_{\psi}$ is a piecewise linear homeomorphism. 
\end{lemma}

\begin{proof}
The pullback is linear so its dual restricts to a piecewise linear map on the dual complexes. Since we are considering a piecewise linear map it suffices to show that $S_{\psi}$ is bijective. But each complex can be viewed as a space of valuations on $K(Y)$. For an element $x \in K(Y)$ and a valuation $v$, we have that $S_{\psi}(v)(x)=v(\psi^*(x))=v(x)$. Thus $S_{\psi}$ is injective.

The points of $D(\Delta_Y)$ with rational coordinates are exactly the divisorial valuations of $D(Y,\Delta_Y)$ with log discrepancy $0$. As $(Y',\Delta_{Y'})$ is a minimal model, these valuations have log discrepancy $0$ for $(Y',\Delta_{Y'})$. Thus $S_{\psi}$ surjects onto this dense set, so it is surjective.
\end{proof}

\begin{proof} [Proof of Thm \ref{main}]
As described in Remark \ref{remark: homotopy equivalence}, we can extend the map $f: X \longrightarrow Y$ to a map on proper snc-models $\mathfrak{f}: \mathcal{X} \longrightarrow \mathcal{Y}$ over $k[[t]]$ that induces a map $S_{\mathfrak{f}}:D(\mathcal{X}_{k})\longrightarrow D(\mathcal{Y}_{k})$. It suffices to construct a section $D(\mathcal{Y}_k)\!\xymatrix{{}\ar@{^{(}->}[r]&{}}\!D(\mathcal{X}_{k})$ of $S_{\mathfrak{f}}$ that is a homotopy inverse of $\mathfrak{f}$.

If $\mathfrak{f}$, $\mathcal{X}$, and $\mathcal{Y}$ are pulled back from a curve, then we can produce a combinatorial section $\mu$ of $f':X \to Y'$ using Theorem \ref{mapdegen} and the fact that $\tau$ is a homotopy equivalence. We have that $S_{\mathfrak{f}}\circ \tau = S_{\psi}$ by Lemma \ref{combinatoricsofspecial} since all these maps are piecewise linear and agree on the $0$-simplices. Finally, $S_{\psi}$ is an isomorphism by Lemma \ref{PLhomeo} so $S_{\mathfrak{f}}$ has a section that is a homotopy equivalence. The theorem follows.

Now for the other case. Any reduced divisor $E_j$ of $Y$ contained in the special fiber pulls back to a divisor $\sum a_{ij} D_i$ in $X$. Choose a positive integer $n$ larger than any of the $a_{ij}$. By Proposition \ref{curvemodel} we can produce a pointed $k$-curve $C\ni p$, along with a map of smooth varieties $g_C: V_C \longrightarrow W_C$ over $C$, such that:
\begin{itemize}
\item[$\bullet$]
The fiber over a general point of $W$ is smooth and rationally connected;
\item[$\bullet$]
The special fiber $V_p$ is snc inside $V$ and $W_p$ is snc inside $W$;
\item[$\bullet$]
There are isomorphisms $\nu_1: (\mathcal{X}_k)_n \longrightarrow (V_p)_n$ and $\nu_2:  (\mathcal{Y}_k)_n \longrightarrow (W_p)_n$ inducing the following commuting square.
\begin{equation}
\xymatrix{
(\mathcal{X}_k)_n\ar[r] \ar[d]& (V_p)_n\ar[d]
\\ (\mathcal{Y}_k)_n\ar[r]& (W_p)_n 
}
\end{equation}
\end{itemize}

By the previous argument, there is a section of $S_{g_C}$
	\begin{equation}\label{equation: section}
	\xymatrix{D(W_p)\ \ar@{^{(}->}[r]&\ D(V_p)}
	\end{equation}
which is a homotopy equivalence. The isomorphisms $\nu_1$ and $\nu_2$ induce isomorphisms between $D(\mathcal{X}_k)$ and $D(V_p)$ and also between $D(\mathcal{Y}_k)$ and $D(W_p)$. The condition on $n$ implies that we get a commuting square 
\begin{equation}
\xymatrix{
D(\mathcal{X}_k)\ar[r] \ar[d]& D(V_p)\ar[d]
\\ D(\mathcal{Y}_k)\ar[r] & D(W_p)
}
\end{equation}
where the vertical arrows are $S_{\mathfrak{f}}$ and $S_{g_C}$ respectively, and the horizontal arrows are isomorphisms. Hence $S_{\mathfrak{f}}$ is a homotopy equivalence, and thus so is $f^{\mathrm{an}}$.

\end{proof}

\vskip 1cm
\bibliographystyle{plain}
\bibliography{rcberkovich}

\end{document}